\documentclass{article}
\usepackage{fancyhdr}
\usepackage{bm}
\usepackage{hyperref}
\usepackage{graphicx} % Required for inserting images
\usepackage{nicematrix}
\usepackage{amsthm,amssymb,amsmath}
\usepackage[T1]{fontenc}
\usepackage[utf8]{inputenc}
\usepackage{hyperref}
\usepackage{lipsum}
\usepackage{mathrsfs}
\usepackage{enumitem}
\usepackage{stmaryrd}
\usepackage{setspace}
\usepackage{yfonts}
\usepackage{float}
\usepackage{mathtools}
\usepackage{parskip}
\usepackage[all,cmtip]{xy}
\usepackage{tikz-cd}
\tikzcdset{row sep/normal=50pt, column sep/normal=50pt}

\newtheorem{lemma}{Lemma}[section]

\newtheorem{theorem}[lemma]{Theorem}

\newtheorem{manualtheoreminner}{Theorem}
\newenvironment{manualtheorem}[1]{%
  \renewcommand\themanualtheoreminner{#1}%
  \manualtheoreminner
}{\endmanualtheoreminner}
\usepackage{blindtext}
\usepackage[a4paper, total={6in, 8in}]{geometry}
\usepackage[backend=biber, style=numeric, sorting=nyt]{biblatex}
\title{Varieties with two smooth blow up structures}
\date{}
\author{Supravat Sarkar\footnote{Princeton University}}
\begin{document}
\maketitle
\begin{abstract}
 We classify smooth projective varieties of Picard rank $2$ which have two structures of blow-up of projective space along smooth subvarieties of different dimensions. This gives a characterization of the so called quadro-cubic Cremona transformation.
\end{abstract}
\begin{center}
\end{center}
\begin{center}
\end{center}
\section{Introduction}
 Given a smooth projective variety $X$, there are two standard ways of constructing another smooth projective variety with Picard number one more than $X$. One is to construct a projective bundle over $X$, another is to blow-up a smooth subvariety of codimension at least $2$ in $X$. It is interesting to consider when a smooth projective variety can be constructed by the above procedure in two different ways, in other words, when a smooth projective variety has two different structures of the above kind.
 
 \cite{sato1985varieties} gives classification of smooth projective varieties having two projective bundle structures over projective spaces. \cite{kanemitsu2019extremal} classifies all smooth projective varieties of Picard rank $2$ having two $\mathbb{P}^1$-bundle structures. \cite{bansal2023isomorphism}, \cite{occhetta2022manifolds}, \cite{occhetta2002euler} also has some results regarding varieties having two projective bundle structures.

 Varieties with both projective bundle and blow up structures are also well-studied in the literature. \cite{vats2024correspondence}, \cite{ray2020examples}, \cite{galkin2022projective}, \cite{bansal2024extremal} gives several examples of such varieties. \cite{li2021projective}, \cite{li2024projective} gives some classification results for such varieties under certain conditions.

 The remaining case, smooth projective varieties having two different smooth blow up structures seems to be comparatively less studied. The goal of this paper is to classify smooth projective varieties of Picard rank $2$ that have two structures of blow-up of projective space along smooth subvarieties of different dimensions.

 This problem has connections to Cremona transformations in the following way. If $p:X\to \mathbb{P}^n $ is a blow-up along $Z_1$, $q:X\to \mathbb{P}^n $ is blow-up along $Z_2$, then we have a birational map $\phi= p \circ q^{-1}: \mathbb{P}^n \dashrightarrow \mathbb{P}^n$, that is, a Cremona transformation of $\mathbb{P}^n$. Cremona transformations which becomes a morphism after blowing up a smooth subvariety were studied in \cite{crauder1989cremona}, \cite{crauder1991cremona}. In \cite{crauder1989cremona}, there is an example of a Cremona transformation $\phi$ of $\mathbb{P}^4$ called \textit{quadro-cubic transformation}. The base locus $Z_2$ of $\phi$ is a (smooth) quintic elliptic curve, and the base locus $Z_1$ of $\phi^{-1}$ is a smooth quintic elliptic scroll. We have a commutative diagram:
 \begin{equation*}
\begin{tikzcd} 
{}& X \arrow[dl, "q"'] \arrow[dr, "p"] & \\
\mathbb{P}^n \arrow[dotted, rr, "\phi"] &&  \mathbb{P}^n 
  & {}  
\end{tikzcd}
\end{equation*}
Here $p$, $q$ are blow-ups along $Z_1$, $Z_2$, respectively. So $X$ is a smooth projective variety of Picard rank $2$ which has two structures of blow-up of projective space along smooth subvarieties of different dimensions. The main theorem of this paper asserts that this is the only example, hence this gives a characterization of quadro-cubic transformation.
 
\begin{manualtheorem}{A}\label{A}
    
Let $n$ be a positive integer, $Z_1,$ $Z_2$ be smooth subvarieties of $\mathbb{P}^n$ with dim $Z_1>$ dim $Z_2$. Let $X$ be a smooth projective variety with two contractions: $p:X\to \mathbb{P}^n $ which is blow-up along $Z_1$, $q:X\to \mathbb{P}^n $ which is blow-up along $Z_2$. Then we have $n=4$, the birational map $\phi=p \circ q^{-1}: \mathbb{P}^n \dashrightarrow \mathbb{P}^n$ is a quadro-cubic transformation, $Z_1=Bs(\phi)$ and $Z_2=Bs(\phi^{-1})$.
 \end{manualtheorem}

Here Bs($\phi$) is the base locus/fundamental locus of $\phi$, defined as follows (see also section $2$ of \cite{ein1989some}): There are homogeneous polynomials $f_0,f_1,...,f_n$ of same degree such that $\phi=(f_0,f_1,...,f_n)$. Define Bs($\phi$) to be the closed subscheme of $\mathbb{P}^n$ cut out by $f_0,f_1,...,f_n$. Similarly bs($\phi^{-1}$) is defined.

Note that our assumptions are less restrictive than the assumptions of \cite{ein1989some}. In \cite{ein1989some} they assume that $\phi$ is a \textit{simple Cremona transformation}, that is, the base locus of $\phi$ is smooth. In our assumptions, the reduced base locus of $\phi$ (or $\phi^{-1}$) is smooth, but the base locus may have some multiplicity $a$ (in Lemma \ref{lemma:duoli} we shall see that the multiplicity is same for the base loci of $\phi$ and $\phi^{-1}$). In fact, we prove the multiplicity to be $1$ in Theorem \ref{theorem:a=1}. In \cite{crauder1991cremona}, they show that if Hartshorne's conjecture on complete intersections is true, then $a=1$. We prove $a=1$ in our setup without assuming Hartshorne's conjecture.

It remains to investigate the case dim$Z_1=$ dim$Z_2$. It turns out that if we assume Hartshorne's conjecture, Corollary $2$ in \cite{crauder1991cremona} gives a satisfactory result towards the classification in this case.

\begin{manualtheorem}{B}
    
 (Corollary $2$ in \cite{crauder1991cremona}):

Let $n,$ $m$ be positive integers, $Z_1,$ $Z_2$ be smooth subvarieties of $\mathbb{P}^n$ with dim $Z_1=$ dim $Z_2=m$. Let $X$ be a smooth projective variety with two different contractions: $p:X\to \mathbb{P}^n $ blow-up along $Z_1$, $q:X\to \mathbb{P}^n $ blow-up along $Z_2$. Let $a$ be the multiplicity of base locus of the birational map $p\circ q^{-1}$, and $c$ be the integer defined by: $\phi$ (or $\phi^{-1}$) is given by an $(n+1)$-tuple of homogeneous polynomials of degree $c$. If Hartshorne's conjecture is true, then $a=1$, and one of the following holds:
\begin{enumerate}
    \item $c=2$, $Z_1\cong Z_2$ are Severi varieties, completely classified in \cite{zak1984varieties},
    \item $c=3$ and either \begin{enumerate}
        \item[(a)] $n=3,$ $m=1$, in this case $\phi$ is a cubo-cubic transformation, as in \cite{crauder1989cremona} theorem $2.2(i)$, or
        \item [(b)] $n=7,$ $ m=4$, or
        \item [(c)] $n=11$, $m=7$, or
        \item [(d)] $n=15$, $m=10$,
    \end{enumerate}
    \item $c=4$ and either \begin{enumerate}
        \item [(a)] $n=9$, $m=6$, or 
        \item [(b)] $n=4$, $m=2$, in this case $\phi$ is a quadro-quadric transformation, completely classified in \cite{crauder1989cremona} theorem $3.3(B)$,
    \end{enumerate}
    \item $c=5$, $n=5$, $m=3$, in this case $Z_1$, $Z_2$ are determinantal varieties, each given by $5\times 5$ minors of a $5\times 6$ matrix of linear forms.
\end{enumerate}
\end{manualtheorem}
\section{Basic equalities and divisibility relations involving intersection numbers}
Throughout this article we work over the field $k=\mathbb{C}$ of complex numbers. A \textit{variety} is an integral, separated scheme of finite type over $k$. We will always be under the following setup.

\textbf{Setup:} Let $X, Z_1, Z_2, p, q, \phi$ be as in the hypothesis of the Theorem A. We refer to the diagram in introduction. Let $E_1=Ex(p),$ $E_2=Ex(q)$ be the exceptional divisors of $p,q$, respectively. Let dim$Z_i=m_i$, deg$Z_i=d_i$ for $i=1,2$. We have $d_1,$ $d_2\geq2$, as blow up of $\mathbb{P}^n$ along a linear subvariety has the other contraction a projective bundle structure, not a blow up. Also, $n-2\geq m_1>m_2\geq 1$ and $n\geq 4$.  Let $F_1$ be a line in a fibre of $p$ over $Z_1,$ $F_2$ be a line in a fibre of $q$ over $Z_2.$ Also let $H_1=p^*\mathcal{O}_{\mathbb{P}^n}(1)$, $H_2=q^*\mathcal{O}_{\mathbb{P}^n}(1)$. We can regard the Cartier divisors $E_1, E_2$ as elements of Pic$(X)$. So, $N_1(X)_{\mathbb{Q}}$ has basis $\{F_1, F_2\}$, and $\{H_1, E_1\}$ and $\{H_2, E_2\}$ are both bases of Pic$(X)\cong \mathbb{Z}^2$. We have $H_1.F_1=H_2.F_2=0$, $E_1.F_1=E_2.F_2=-1$. Let $a=H_1.F_2$, $b=H_2.F_1$, $c=E_1.F_2$, $d=E_2.F_1$.

Similar to \cite{li2021projective}, we have
\begin{lemma}\label{lemma:duoli}
    \begin{enumerate}
        \item [(i)] $a=b>0$, $c,d>0$, $E_1\neq E_2$.
        \item [(ii)] $a|cd-1$, and we have

        $$H_1=dH_2-aE_2$$

        $$E_1=\frac{cd-1}{a}H_2-cE_2$$

        in Pic($X$).
    \end{enumerate}
\end{lemma}
\begin{proof}
Let $A\in GL_2(\mathbb{Z})$ be such that $\begin{pmatrix}
    H_1 \\
    E_1
\end{pmatrix}=A \begin{pmatrix}
    H_2 \\
    E_2
\end{pmatrix}$. The intersection numbers imply: $$\begin{pmatrix}
    0 & a \\
    -1 & c
\end{pmatrix}=A \begin{pmatrix}
    b & 0 \\
    d & -1
\end{pmatrix}.$$
Taking determinants: $a=-b(detA)$. Note that $a=$ deg $p_*F_2>0$, $b=$ deg $q_*F_1>0$. This forces det$A=-1$, $a=b>0$. If $c\leq0$, $-E_1$ would be nef. But this is impossible as $E_1$ is effective. So, $c>0$. Similarly, $d>0$. If $E_1=E_2$, then $c=E_1.F_2=E_2.F_2=-1<0$, a contradiction. So, $E_1\neq E_2$. This proves $(i)$.

For $(ii)$, note that $A=\begin{pmatrix}
    0 & a \\
    -1 & c
\end{pmatrix} \begin{pmatrix}
    a & 0 \\
    d & -1
\end{pmatrix}^{-1}=\begin{pmatrix}
    d & -a \\
    \frac{cd-1}{a} & -c
\end{pmatrix}$.
\end{proof}
\begin{lemma}\label{lemma:katz}
    \begin{enumerate}
        \item [(i)] $(d-1)(n+1)=(n-m_1-1)\frac{cd-1}{a}$,
        \item [(ii)] $(c-1)(n+1)=(n-m_2-1)\frac{cd-1}{a}$,
        \item [(iii)] $\frac{c}{a}+\frac{1}{a}. \frac{n-m_2-1}{n-m_1-1}=\frac{n+1}{n-m_1-1}$,
        \item [(iv)]$\frac{d}{a}+\frac{1}{a} .\frac{n-m_1-1}{n-m_2-1}=\frac{n+1}{n-m_2-1}$,
        \item [(v)]$c>d\geq 2$.
    \end{enumerate}
\end{lemma}
\begin{proof}
    We have $K_X=p^*K_{\mathbb{P}^n}+(n-m_1-1)E_1=-(n+1)H_1+(n-m_1-1)E_1$. Similarly, $K_X=-(n+1)H_2+(n-m_2-1)E_2$. Equating them and using Lemma \ref{lemma:duoli} $(ii)$, we get $(i)$ and $(iii)$. By symmetry $(ii)$ and $(iv)$ follows.

    Since $m_1>m_2$, $(i)$ and $(ii)$ implies $c>d$. If $d=1$, $(i)$ shows $cd=1$, so $c=1$, contradiction to $c>d$. So, $d\geq 2$. This proves $(v)$.
\end{proof}
\begin{lemma}\label{lemma:leray}
If $E$ is a vector bundle of rank $r+1$ on a smooth projective variety $X$, then under the natural pullback we have,$\vspace{1mm}$ $A^{\ast}(\mathbb{P}(E))\cong {A^{\ast}(X)[u]}/(\sum_{i = 0}^{r+1}(-1)^{i}{c_{i}(E)u^{r+1-i}})$ as $A^{\ast}(X)$-algebra. The isomorphism is given $\vspace{1mm}$ by $[\mathcal{O}_{\mathbb{P}(E)}(1)]\longleftarrow u$.
\end{lemma}
\begin{proof}
See \cite{lazarsfeld2017positivity}, Appendix A.
\end{proof}
\begin{lemma}\label{lemma:EH}
    \begin{enumerate}
        \item [(i)] $E_2^i H_2^{n-i}=0$ if $0<i<n-m_2$, $E_2^{n-m_2} H_2^{m_2}=(-1)^{n-m_2-1}d_2$,
        \item [(ii)] $E_1^i H_1^{n-i}=0$ if $0<i<n-m_1$, $E_1^{n-m_1} H_1^{m_1}=(-1)^{n-m_1-1}d_1$,
        \item[(iii)] $a^{n-m_2}|cd-1$,
        \item[(iv)] $d_2<(\frac{d}{a})^{n-m_2}$.
    \end{enumerate}
\end{lemma}
\begin{proof}
 $(i)$ and $(ii)$ are proven in the 1st paragraph of the proof of  Formulae 0.3 in \cite{crauder1989cremona}. $(iv)$ is Formula 0.3.v in \cite{crauder1989cremona}. We prove $(iii)$.

 We have $$\hspace{-260pt}0=E_1H_1^{n-1}\text{   (by } (ii))$$ $$\hspace{-180pt}=(\frac{cd-1}{a}H_2-cE_2)(dH_2-aE_2)^{n-1}$$ $$=(\frac{cd-1}{a}H_2-cE_2)(d^{n-1}H_2^{n-1}+\sum_{i=0}^{m_2}{{n-1}\choose i}d^i(-a)^{n-1-i}H_2^i E_2^{n-1-i})\text{  (by }(i))$$ $$\hspace{-200pt}\equiv \frac{(cd-1)d^{n-1}}{a}(\text{mod } a^{n-m_2-1}).$$ So, $a^{n-m_2}|(cd-1)d^{n-1}$. Since $a|cd-1$, we have gcd$(a,d)=1$. So, $a^{n-m_2}|cd-1$.
\end{proof}
\begin{lemma} \label{lemma:estimate}
    \begin{enumerate}
        \item [(i)] $(n-m_1-1)(n-m_2-1)a^{n-m_2-1}|a(n+1)^2-(n+1)(2n-2-m_1-m_2)$, $a(n+1)^2-(n+1)(2n-2-m_1-m_2)>0$,
        \item[(ii)] $(n+1)^2>a^{n-m_2-2}(n-m_2-1)(n-m_1-1)\geq a^{n-m_1-1}(n-m_1)(n-m_1-1)$. 
    \end{enumerate}
\end{lemma}
\begin{proof}
    $(i)$: Using Lemma \ref{lemma:katz} $(iii)$ and $(iv)$, we get $$cd-1=\frac{a^2(n+1)^2-a(n+1)(2n-2-m_1-m_2)}{(n-m_1-1)(n-m_2-1)}.$$ Since $c>d\geq 2$ by Lemma \ref{lemma:katz} $(v)$, we see that $cd-1>0$. So, $a(n+1)^2-(n+1)(2n-2-m_1-m_2)>0$. Now Lemma \ref{lemma:EH} $(iii)$ implies $(i)$. 

    $(ii)$ follows from $(i)$, by noting that $m_1+m_2<2n-2$ as $m_1, m_2<n-1$ and $m_2+1\leq m_1$.
\end{proof}
For a compact complex manifold $Y$ and an integer $i$, denote dim$_{\mathbb{C}}$$H^{i}(Y, \mathbb{C})$ by $h^i(Y)$.
Let $a_i=h^{2i}(Z_1)$, $b_i=h^{2i}(Z_2)$.
\begin{lemma}\label{lemma:betti}
    \begin{enumerate}
        \item [(i)] $a_i\neq 0 \Longleftrightarrow 0\leq i\leq m_1$, $b_i\neq 0 \Longleftrightarrow 0\leq i\leq m_2$,
        \item[(ii)] $a_i-a_{i-(n-m_1-1)}=b_i-b_{i-(n-m_1-1)}$,
        \item[(iii)] $a_i=b_i$ if $0\leq i\leq n-m_1-2$,
        \item[(iv)] If $m_1\geq \frac{3n-2}{4}$, then $m_2\leq n-m_1-2$. 
    \end{enumerate}
\end{lemma}
\begin{proof}
 $(i)$ follows from the general fact that for a compact Kahler manifold $Y$ of complex dimension $m$, we have $H^{2i}(Y,\mathbb{C})\neq 0$ for $0\leq i\leq m$.
 
 $(ii)$: By {\cite[Theorem 7.31]{MR2451566}}, $$H^k(X,\mathbb{Z})\cong H^k(\mathbb{P}^n,\mathbb{Z})\oplus(\oplus_{i=0}^{n-m_1-2}H^{k-2i-2}(Z_1,\mathbb{Z})).$$
 So, for any $k$ we have $$h^{2k}(X)=\sum_{i=0}^{n-m_1-2}h^{2k-2i-2}(Z_1)+h^{2k}(\mathbb{P}^n)=\sum_{i=0}^{n-m_1-2}a_{k-i-1}+h^{2k}(\mathbb{P}^n).$$ Similarly, $$h^{2k}(X)=\sum_{i=0}^{n-m_2-2}b_{k-i-1}+h^{2k}(\mathbb{P}^n).$$ So we obtain
 \begin{equation}
     \sum_{i=0}^{n-m_1-2}a_{k-i-1}=\sum_{i=0}^{n-m_2-2}b_{k-i-1}.
 \end{equation}
 Replacing $k$ by $k+1$ in $(1)$ we get
 \begin{equation}
     \sum_{i=-1}^{n-m_1-3}a_{k-i-1}=\sum_{i=-1}^{n-m_2-3}b_{k-i-1}.
 \end{equation}
 Now $(2)-(1)$ gives $a_k-a_{k-(n-m_1-1)}=b_k-b_{k-(n-m_1-1)}$ for all $k$. This proves $(ii)$.

 $(iii)$ follows from $(ii)$ as $i\leq n-m_1-2\implies i-(n-m_1-1), i-(n-m_2-1)<0$, so $a_{i-(n-m_1-1)}=b_{i-(n-m_2-1)}=0$.
 
    $(iv)$: If $m_1\geq \frac{3n-2}{4}$, we have $2(n-m_1-1)\leq 2m_1-n$. By Barth-Larsen theorem ( \cite[Theorem 3.2.1(i)]{lazarsfeld2017positivity}), $a_{n-m_1-1}=h^{2(n-m_1-1)}(Z_1)=1$. By $(ii)$, $b_{n-m_1-1}=a_{n-m_1-1}-1=0$. So by $(i)$, $m_2\leq n-m_1-2$.
\end{proof}
\begin{lemma}\label{lemma:HC}
    If $m_2\leq 2n/3$, then $a=1$.
\end{lemma}
\begin{proof}
    This follows from Theorem 1 of \cite{crauder1991cremona}, applied to the Cremona transformation $\phi.$
\end{proof}
\section{Proof of Theorem A}
First we prove the following.
\begin{theorem}\label{theorem:a=1}
In the notations of the setup, we have $a=1.$
\end{theorem}
\begin{proof}
    Suppose $a\geq 2$. We wish to get a contradiction. If $m_1\geq\frac{3n-2}{4}$, then by Lemma \ref{lemma:betti} $(iv)$, $m_2\leq n-m_1-2$. On the other hand, by Lemma \ref{lemma:HC}, $m_2>2n/3$. So, $2n/3<n-m_1-2\leq n-2-\frac{3n-2}{4}$, so $2n/3<\frac{n-6}{4}$, implying $5n<-18$, which is absurd. So, $m_1<\frac{3n-2}{4}$.

    We also have $m_1\geq m_2+1>2n/3+1$ by Lemma \ref{lemma:HC}. So, $2n/3+1<\frac{3n-2}{4}$, implying $n>18$.

 Note that $m_1<\frac{3n-2}{4} \implies n-m_1>\frac{n+2}{4}\implies n-m_1\geq \lceil{\frac{n+2}{4}} \rceil$. By Lemma \ref{lemma:estimate} $(ii)$, $(n+1)^2>a^{n-m_1-1}(n-m_1)(n-m_1-1)\geq a^{\lceil{\frac{n-2}{4}}\rceil}.\lceil\frac{n-2}{4}\rceil.\lceil\frac{n+2}{4}\rceil\geq 2^{\lceil{\frac{n-2}{4}}\rceil}.\lceil{\frac{n-2}{4}}\rceil. \lceil{\frac{n+2}{4}}\rceil$. So, $$(n+1)^2> 2^{\lceil{\frac{n-2}{4}}\rceil}.\left\lceil{\frac{n-2}{4}}\right\rceil. \left\lceil{\frac{n+2}{4}}\right\rceil.$$ One easily sees that this is not satisfied if $n>18$.

 This contradiction shows $a=1$.
    
\end{proof}

Now we show that the $6-$tuple $(n,a,c,d,m_1,m_2)$ can have at most $2$ possibilities.
\begin{theorem}\label{theorem:2case}
In the notations of the setup, we have
    \begin{enumerate}
        \item [either (1)] $n=4, a=1, c=3, d=2, m_1=2, m_2=1$,
     \item [or (2)] $n=9, a=1, c=3, d=2, m_1=6, m_2=4$.
    \end{enumerate}
\end{theorem}
\begin{proof}
    By Theorem \ref{theorem:a=1}, $a=1$. If $m_1\geq\frac{3n-2}{4}$, by Lemma \ref{lemma:betti} $(iv)$, $m_1+m_2\leq n-2$, and we have $$2\leq d<d+\frac{n-m_1-1}{n-m_2-1}=\frac{n+1}{n-m_2-1}\text{ (by Lemma \ref{lemma:katz}} (iv))$$ $$\leq\frac{n+1}{m_1+1}\leq\frac{n+1}{\frac{3n-2}{4}+1}=\frac{4n+4}{3n+2}.$$
    So, $$4n+4>2(3n+2)=6n+4\implies n<0,$$ a contradiction. So, $m_2<m_1<\frac{3n-2}{4}$.

    By Lemma \ref{lemma:katz} $(iii), (iv)$,
    \begin{equation}
        c=\frac{m_2+2}{n-m_1-1}, d=\frac{m_1+2}{n-m_2-1}.
    \end{equation}
    \textbf{Claim:} $c=3$, $d=2$.
    \begin{proof}
        By Lemma \ref{lemma:katz} $(v)$, it suffices to show $c\leq 3$. Suppose $c\geq 4$. By $(3)$, $m_2+2\geq 4(n-m_1-1)$, so $4n-6\leq m_2+4m_1<\frac{3n-2}{4}-1+3n-3$ (as $m_2\leq m_1-1<\frac{3n-2}{4}-1, 4m_1<3n-2\implies 4m_1\leq 3n-3$) $=15n/4-9/2$.

        Hence, $4n-6<15n/4-9/2$, that is, $n\leq 5$.

        If $n=5$, then $(m_2,m_1)=(1,2), (1,3)$ or $(2,3)$. In each case either $c$ or $d$ is not an integer, a contradiction.

        If $n=4$, we must have $m_1=2, m_2=1$. So, $c=3$ by $(3)$. Contradiction to $c\geq 4$.

        This contradiction shows $c\leq 3$.
    \end{proof}

    Now $(3)$ and claim gives 
    \begin{equation}
        m_1=\frac{4n-6}{5}, m_2=\frac{3n-7}{5}.
    \end{equation}
    Since $m_1, m_2$ are integers, we get $n\equiv -1(\text{mod } 5)$. Also, $$m_1<\frac{3n-2}{4}\implies \frac{4n-6}{5}<\frac{3n-2}{4}\implies n<14.$$ Together with $n\equiv -1(\text{mod } 5)$, we get $n=4$ or $9$.
    If $n=4$, then $m_1=2, m_2=1$ by $(4)$, so case $1)$ occurs. If $n=9$, then $m_1=6, m_2=4$ by $(4)$, so case $2)$ occurs. 
\end{proof}
\begin{proof}[Proof of Theorem \ref{A}]
By Theorem \ref{theorem:2case} we are in case $1)$ or $2)$. If we are in case $1)$, then $\phi, \phi^{-1}$ are simple Cremona transformations with dim Bs$(\phi)=1$, dim Bs$(\phi^{-1})=2$. By \cite{crauder1989cremona}, $\phi$ is a quadro-cubic transformation. So it suffices to show that case $2)$ cannot occur.

Suppose case $2)$ occurs. We want to derive a contradiction. Let $x=H_2^3E_2^6,$ $ y=H_2^2E_2^7,$ $ z=H_2E_2^8,$ $ w=E_2^8.$

\textbf{Claim 1:} $\begin{pmatrix}
    x \\
    y \\
    z \\
    w
\end{pmatrix}=-\begin{pmatrix}
    37-12d_2+d_1 \\
    399-84d_2+14d_1 \\
    2493-448d_2+112d_1 \\
    11771-2016d_2+672d_1
\end{pmatrix}$
\begin{proof}
    Using Lemma \ref{lemma:EH} $(i), (ii)$, the equations $$1=H_1^9=(2H_2-E_2)^9,$$ $$0=H_1^8E_1=(2H_2-E_2)^8(5H_2-3E_2),$$ $$0=H_1^7E_1^2=(2H_2-E_2)^7(5H_2-3E_2)^2,$$ $$d_1=H_1^6E_1^3=(2H_2-E_2)^6(5H_2-3E_2)^3,$$ we get
    $$672x-144y+18z-w=-511+2016d_2,$$ $$1904x-416y+53z-3w=-1280+5600d_2,$$ $$5390x-1201y+156z-9w=-3200+15540d_2,$$ $$15245x-3465y+459z-27w=-8000+43080d_2+d_1.$$
In matrix notations,
 $\begin{pmatrix}
    672 & -144 & 18 & -1 \\
    1904 & -416 & 53 & -3 \\
    5390 & -1201 & 156 & -9 \\
    15245 & -3465 & 459 & -27
\end{pmatrix}\begin{pmatrix}
    x \\
    y \\
    z \\
    w
\end{pmatrix}=-\begin{pmatrix}
    511-2016d_2 \\
    1280-5600d_2 \\
    3200-15540d_2 \\
    8000-43080d_2-d_1
\end{pmatrix}$. So, $$\begin{pmatrix}
    x \\
    y \\
    z \\
    w
\end{pmatrix}=-\begin{pmatrix}
    27 & -27 & 9 & -1 \\
    369 & -372 & 125 & -14 \\
    2883 & -2929 & 992 & -112 \\
    16901 & -17298 & 5904 & -672
\end{pmatrix}\begin{pmatrix}
    511-2016d_2 \\
    1280-5600d_2 \\
    3200-15540d_2 \\
    8000-43080d_2-d_1
\end{pmatrix}$$ $$=-\begin{pmatrix}
    37-12d_2+d_1 \\
    399-84d_2+14d_1 \\
    2493-448d_2+112d_1 \\
    11771-2016d_2+672d_1
\end{pmatrix}$$.
\end{proof}
\textbf{Claim 2:}
 $b_i=1$ for all $0\leq i \leq 4$.
 \begin{proof}
     By Lemma \ref{lemma:betti} $(ii)$,
     \begin{equation}
         a_i-a_{i-2}=b_i-b_{i-4}.
     \end{equation}
     By Barth-Larsen theorem(Theorem $3.2.1(i)$ in \cite{lazarsfeld2017positivity}), $a_1=1$. So by Lemma \ref{lemma:betti} $(iii)$, $b_1=1$. By Poincare duality, $a_5=b_3=1$. Clearly, $a_0=a_6=b_0=b_4=1.$ By $(5)$, $a_2=1+b_2$, $a_3=1+b_3=2$. By Hard Lefschetz theorem (Corollary $3.1.40$ in \cite{lazarsfeld2017positivity}), the map $H^4(Z_1,\mathbb{C})\to H^6(Z_1,\mathbb{C})$ given by multiplication by a Kahler class in $H^2(Z_1,\mathbb{C})$ is injective, so $a_2\leq a_3$. This shows $b_2\leq 1.$ Since $b_2\neq 0$ by Lemma \ref{lemma:betti} $(i)$, we get $b_2=1$.
 \end{proof}

 By abuse of notation denote $c_1(\mathcal{O}_{Z_2}(1))\in H^2(Z_2,\mathbb{Z})$ by $H_2$. So $H_2^i$ is a nonzero element of $H^{2i}(Z_2,\mathbb{Z})/torsion\cong \mathbb{Z} $ for $1\leq i\leq 4.$ So there are $l_i\in \mathbb{N}$, $i=1,2,3,4$ such that $H_2^i/l_i$ is a generator of $H^{2i}(Z_2,\mathbb{Z})/torsion$. Let $\alpha=l_1,$ $k=l_2$. We know $l_4=d_2$. By Poincare duality, $H_2/l_1 . H_2^3/l_3=\pm H_2^4/l_4$, $(H_2^2/l_2)^2=\pm H_2^4/l_4$. This forces $l_1 l_3=l_4=l_2^2$. So, $l_3=k^2/\alpha$, $d_2=k^2$. Also, $(H_2/\alpha)^2=H_2^2/\alpha^2$ is an integral multiple of $H_2^2/k$, so $\alpha^2|k$. Write $k=\alpha^2\beta,$ $\beta\in \mathbb{N}$. So, $d_2=\alpha^4\beta^2$.
 
 \textbf{Claim 3:} \begin{enumerate}
     \item [(i)] $\alpha^3\beta^2|x$, $\alpha^2\beta|y$.
     \item [(ii)] $\alpha=1$, $\beta=7,17$ or $7\times 17$.
 \end{enumerate}
 \begin{proof}
     $(i):$ Let $N_2$ be the normal bundle of $Z_2$ in $\mathbb{P}^9$, so $E_2=\mathbb{P}(N_2^*)$. Let $c_1,c_2,c_3,c_4\in \mathbb{Q}$ be such that $c_i(N_2^*)=c_iH_2^i$ in $H^*(Z_2,\mathbb{Q})$.  Let $v=c_1(\mathcal{O}_{\mathbb{P}(N_2^*)}(1))\in H^2(E_2,\mathbb{Z})$. By Lemma \ref{lemma:leray}, $$v^5-c_1H_2v^4+c_2H_2^2v^3-c_3H_2^3v^2+c_4H_2^4v=0$$ in $A^*(E_2)\otimes \mathbb{Q}$, hence in $H^*(E_2,\mathbb{Q}) $ by the natural ring homomorphism $A^*(E_2)\otimes \mathbb{Q} \to H^*(E_2,\mathbb{Q}) $ (\cite{fulton2013intersection}, chapter 19). Here $H^*(E_2,\mathbb{Q})$ is considered as a $H^*(Z_2,\mathbb{Q})$-algebra by $q^*$. We have $x=H_2^3E_2^6=-v^5H_2^3=-c_1v^4H_2^4=-c_1E_2^5H_2^4=-c_1d_2$. So, 
     \begin{equation}
         c_1=-x/d_2.
     \end{equation}
     We also have $y=H_2^2E_2^7=v^6H_2^2=c_1v^5H_2^3-c_2v^4H_2^4=-c_1x-c_2d_2=x^2/d_2-c_2d_2.$ So,
     \begin{equation}
         c_2=x^2/d_2^2-y/d_2.
     \end{equation}
     Since $c_i(N_2^*)$ is an integral multiple of $H_2^i/l_i$, we get $c_1\alpha, c_2\alpha^2\beta\in\mathbb{Z}$. So by $(6)$, $\mathbb{Z}\ni \frac{\alpha x}{d_2}=\frac{\alpha x}{\alpha^4\beta^2}=\frac{x}{\alpha^3\beta^2} $. So, $\alpha^3\beta^2|x.$
     
     By $(7)$, $\mathbb{Z}\ni \frac{\alpha^2\beta x^2}{\alpha^8\beta^4}-\frac{\alpha^2\beta y}{\alpha^4\beta^2}=\frac{x^2}{\alpha^6\beta^3}-\frac{y}{\alpha^2\beta}$.

     Since $\alpha^3\beta^2|x$, we have $\alpha^6\beta^3|\alpha^6\beta^4|x^2$. So, $\alpha^2\beta|y$. This proves $(i)$.

     $(ii):$ By $(i)$, $\alpha^2\beta|x,y,d_2$. So by Claim 1, $d_1\equiv -37(\text{mod } \alpha^2\beta)$, $14d_1\equiv-399(\text{mod } \alpha^2\beta)$. So, $399\equiv-14d_1\equiv 14\times 37=518(\text{mod } \alpha^2\beta) $. Hence $\alpha^2\beta|518-399=119=7\times 17$. This forces $\alpha=1$. So $d_2=\alpha^4\beta^2=\beta^2$. If $\beta=1$, then $d_2=1$, a contradiction. So, $\beta=7,17$ or $7\times 17$.
 \end{proof}

 Now we are ready to get a contradiction. By Lemma \ref{lemma:EH} $(iv)$, $d_2<32$. But by Claim 3 $(ii)$, $d_2=\beta^2\geq 7^2=49$, a contradiction.
\end{proof}
\section{Acknowledgement}
The author is grateful to Professor János Kollár for giving valuable references.

\printbibliography
\end{document}